%% file: main.tex
\documentclass[a4paper,11pt]{article}
\usepackage{amsmath}
\usepackage{amsfonts}
\usepackage{mathtools}
\usepackage{xspace}
\usepackage{graphicx}
\usepackage{color}
\usepackage{subfig}
\usepackage{tikz}
\usepackage{geometry}
\usepackage{pdflscape}
\usepackage{afterpage}
\usepackage{url}
\usepackage{capt-of}
\usepackage{lipsum}
\usetikzlibrary{positioning}
\usepackage{algpseudocode}
\usepackage{algorithm} 
\usepackage{amsthm}
\usepackage[T1]{fontenc}
\newtheorem{theorem}{Theorem}

\input alias

\title{A quadratic penalty algorithm for linear programming and its application to linearizations of quadratic assignment problems}
\date{University of Edinburgh\\School of Mathematics and Maxwell Institute for Mathematical Sciences\\James Clerk Maxwell Building\\Peter Guthrie Tait Road, Edinburgh, EH9 3FD, UK\\Email J.A.J.Hall@ed.ac.uk\\[\baselineskip] \today}
\author{
  I. L. Galabova 
  \and
  J. A. J. Hall}

\begin{document}
\maketitle

\begin{abstract}
This paper provides the first meaningful documentation and analysis of an established technique which aims to obtain an approximate solution to linear programming problems prior to applying the primal simplex method. The underlying algorithm is a penalty method with naive approximate minimization in each iteration. During initial iterations an approach similar to augmented Lagrangian is used. Later the technique corresponds closely to a classical quadratic penalty method. There is also a discussion of the extent to which it can be used to obtain fast approximate solutions of LP problems, in particular when applied to linearizations of quadratic assignment problems.
\end{abstract}


\section{Introduction}\label{sect:Intro}

The efficient solution of linear programming (LP) problems is crucial for a wide range of practical applications, both as problems modelled explicitly, and as subproblems when solving discrete and nonlinear optimization problems. Finding an approximate solution particularly rapidly is valuable as a ``crash start'' to an exact solution method. There are also applications where it is preferable to trade off solution accuracy for a significant increase in speed. 

The ``Idiot'' crash within the open source LP solver \clp~\cite{ClpURL} aims to find an approximate solution of an LP problem prior to applying the primal revised simplex method. In essence, the crash replaces the minimization of the linear objective subject to linear constraints by the minimization of the objective plus a multiple of a quadratic function of constraint violations. 

Section~\ref{sect:Background} sets out the context of the Idiot crash within~\clp and the very limited documentation and analysis which exists. Since the Idiot crash is later discussed in relation to the quadratic penalty and augmented Lagrangian methods, a brief introduction to these established techniques is also given. The algorithm used by the Idiot crash is set out in Section~\ref{sect:IdiotDefPlusResultsPlusTheory}, together with results of experiments on representative LP test problems and a theoretical analysis of its properties. The extent to which the Idiot crash can be used to obtain fast approximate solutions of LP problems, in particular when applied to linearizations of quadratic assignment problems (QAPs), is explored in Section~\ref{sect:FASLP}. Conclusions are set out in Section~\ref{sect:Conclusions}.

\section{Background}\label{sect:Background}

For convenience, discussion and analysis of the algorithms in this paper are restricted to linear programming (LP) problems in standard form
\begin{equation}
\textrm{minimize}~f=\bfc^T \bfx\quad\textrm{subject to}~A\bfx = \bfb\quad\bfx \ge \bfzero,
\label{eq:standardForm}
\end{equation}
\noindent where $\bfx\in\R^n$, $\bfc\in\R^n$, $A\in\R^{m \times n}$, $\bfb\in\R^m$ and $m<n$. In problems of practical interest, the number of variables and constraints can be large and the matrix $A$ is sparse. It can also be assumed that $A$ has full rank of $m$. The algorithms, discussion and analysis below extend naturally to more general LP problems. 

The Idiot crash was introduced into the open source LP solver \clp~\cite{ClpURL} in 2002 and aims to find an approximate solution of an LP problem prior to applying the primal revised simplex method. Beyond its definition as source code~\cite{IdiotComment_ClpSource}, a 2014 reference by Forrest to having given ``a bad talk on it years ago''~\cite{IdiotTalk}, as well as a few sentences of comments in the source code, documentation for the mixed-integer programming solver \cbc~\cite{IdiotComment_GAMS_CBC}, and a public email~\cite{IdiotTalk}, the Idiot crash lacks documentation or analysis. Forrest's comments stress the unsophisticated nature of the crash and only hint at its usefulness as a means of possibly obtaining an approximate solution of an LP problem prior to applying the primal simplex algorithm. However, for several test problems used in the Mittelmann benchmarks~\cite{MittelmannBenchmarksURL}, \clp is significantly faster than at least one of the three major commercial solvers (\cplex, \gurobi and \xpress), and experiments in Section~\ref{subsect:PreliminaryExperiments} show that the Idiot crash is a major factor in this relative performance. For three of these test problems, \LP{nug12}, \LP{nug15} and \LP{qap15}, which are quadratic assignment problem (QAP) linearizations, the Idiot crash is shown to be particularly effective. This serves as due motivation for studying the algorithm, understanding why it performs well on certain LP problems, notably QAPs, and how it might be of further value.

The Idiot crash terminates at a point which, other than satisfying the bounds $\bfx\ge\bfzero$, has no guaranteed properties. In particular, it satisfies no known bound on the residual $\|A\bfx-\bfb\|_2$ or distance (positive or negative) from the optimal objective value. Although some variables may be at bounds, there is no reason why the point should be a vertex solution. Thus, within the context of \clp, before the primal simplex method can be used to obtain an optimal solution to the LP problem, a ``crossover'' procedure is required to identify a basic solution from the point obtained via the Idiot crash. 

\subsection{Penalty function methods}

Although Forrest states that the Idiot crash minimizes a multiple of the LP objective plus a sum of squared primal infeasibilities~\cite{IdiotTalk}, a more general quadratic function of constraint violations is minimized. This includes a linear term, making the Idiot objective comparable with an augmented Lagrangian function. For later reference, these two established penalty function methods are outlined below.

\subsubsection*{The quadratic penalty method}

For the nonlinear equality problem
\begin{equation}
\label{eq:NLP_Equality_Problem}
\textrm{minimize}~f(\bfx)\quad\textrm{subject to}~\bfr(\bfx) = \bfzero,
\end{equation}
the quadratic penalty method minimizes 
\begin{equation}
\phi(\bfx,\mu)=f(\bfx)+\frac{1}{2\mu}\bfr(\bfx)^T\bfr(\bfx),
\label{eq:QuadraticPenaltyFunction_for_NLP}
\end{equation}
for an increasing sequence of positive values of $\mu$. If $\bfx^k$ is the global minimizer of $\phi(\bfx,\mu^k)$ and $\mu^k\downarrow 0$ then Nocedal and Wright~\cite{NoWr06} show that every limit point $\bfx^*$ of the sequence $\{\bfx^k\}$ is a global solution of~(\ref{eq:NLP_Equality_Problem}). The subproblem of minimizing $\phi(\bfx,\mu^k)$ is known to be increasingly ill-conditioned as smaller values of $\mu^k$ are used~\cite{NoWr06} and this is one motivation for the use of the augmented Lagrangian method.

\subsubsection*{The augmented Lagrangian method} 
The augmented Lagrangian method, outlined in Algorithm~\ref{alg_AL_nlp}, was originally presented as an approach to solving nonlinear programming problems. It was first proposed by Hestenes in his survey of multiplier and gradient methods~\cite{Hestenes69} and then fully interpreted and analysed, first by Powell~\cite{Powell69} and then by Rockafellar~\cite{Rockafellar74}. The augmented Lagrangian function~\eqref{eq:AL_for_NLP} is in a sense a combination of the Lagrangian function and the quadratic penalty function~\cite{NoWr06}. It is the quadratic penalty function with an explicit estimate of the Lagrange multipliers $\bflambda$.

\begin{equation}
\mathcal{L}_A(\bfx, \bflambda, \mu) = f(\bfx)  + \bflambda^T \bfr(\bfx) + \dfrac{1}{2\mu}\bfr(\bfx)^T\bfr(\bfx)
\label{eq:AL_for_NLP}
\end{equation}

\begin{algorithm}
    \caption{The augmented Lagrangian algorithm}
    \begin{algorithmic}[1]
	\Statex{ Initialize $\bfx^0 \geq \bf0$, {$\mu^0$}, $\boldsymbol{\lambda}^0  $ and a tolerance $\tau^0$ } 
	\Statex{ For $k=0,1,2,...$ }
	\Statex{ $\quad \quad$ Find an approximate minimizer $\bfx^k$ of $\mathcal{L}_A(\dot , \bflambda^k, \mu^k)$, starting at $\bfx^k$}
	\Statex{ $\quad \quad \quad \quad$ and terminating when $\| \nabla_{\bfx} \mathcal{L}_A(\bfx^k , \bflambda^k, \mu^k) \| \leq \tau^k$}
	\Statex{ $\quad\quad$ If a convergence test for \eqref{eq:NLP_Equality_Problem} is satisfied}
	\Statex{ $\quad\quad\quad\quad$ \textbf{stop} with an approximate solution $\bfx^k$}
	\Statex{ $\quad\quad$ End if } 
	\Statex{ $\quad\quad$ Update Lagrange multipliers $\bflambda$}
	\Statex{ $\quad\quad\quad\quad$ Set $\bflambda_i^{k+1} = \bflambda_i^{k} + \mu^k \bfr(\bfx^k)$}
	\Statex{ $\quad\quad$ Choose new penalty parameter $\mu^{k+1}$ so that $0 \leq \mu^{k+1} \leq \mu^k$}
	\Statex{ $\quad\quad$ Choose new tolerance $\tau^{k+1}$}
	\Statex{ End}
    \end{algorithmic}
    \label{alg_AL_nlp}
\end{algorithm}

Although originally intended for nonlinear programming problems, the augmented Lagrangian method has also been applied to linear programming problems~\cite{EvGoMo05, Gu92}. However, neither article assesses its performance on large scale practical LP problems.

\section{The Idiot crash}\label{sect:IdiotDefPlusResultsPlusTheory}

This section presents the Idiot crash algorithm, followed by some practical and mathematical analysis of its behaviour. Experiments assess the extent to which the Idiot crash accelerates the time required to solve representative LP test problems using the primal revised simplex method, and can be used to find a feasible and near-optimal solution of the problems. Theoretical analysis of the limiting behaviour of the algorithm shows that it will solve any LP problem which has an optimal solution.

\subsection{The Idiot crash algorithm}

The Idiot crash algorithm minimizes the function 
\begin{equation}
h(\bfx) = \bfc^T\bfx + \bflambda^T\bfr(\bfx) + \dfrac{1}{2\mu} \bfr(\bfx)^T\bfr(\bfx),\quad \textrm{where}\quad \bfr(\bfx)=A\bfx-\bfb
\label{eq:IdiotFunction}
\end{equation}
subject to the bounds $\bfx\ge\bfzero$ for sequences of values of parameters $\bflambda$ and $\mu$. The minimization is not performed exactly, but approximately, by minimizing with respect to each component of $\bfx$ in turn. The general structure of the algorithm is set out in Algorithm~\ref{alg_penaltyIdiot}. The number of iterations is determined heuristically according to the size of the LP and progress of the algorithm. Unless the Idiot crash is abandoned after up to around $20$ ``sample'' iterations (see below), the number of iterations performed ranges between 30 and 200. The value of $\mu^0$ ranges between $0.001$ and $1000$, again according to the LP dimensions. When $\mu$ is changed, the factor by which it is reduced is typically $\omega=0.333$. The final value of $\mu$ is typically a little less than machine precision.

\begin{algorithm}
    \caption{The Idiot penalty algorithm}
    \begin{algorithmic}[1]
	\Statex{ Initialize $\bfx^0\ge\bfzero$, {$\mu^0$}, $\boldsymbol{\lambda}^0 = \bf0$ 
    } 
	\Statex{ Set $\mu^1=\mu^0$ and $\boldsymbol{\lambda}^1=\boldsymbol{\lambda}^0$}
	\Statex{ For $k=1,2,3,...$ }
	\Statex{ $\quad \quad \bfx^k = \displaystyle \arg \min_{\bfx\ge\bfzero} h(\bfx) = \bfc^T\bfx + 
	\boldsymbol{\lambda}^T\bfr(\bfx) + \dfrac{1}{2\mu} \bfr(\bfx)^T\bfr(\bfx)$}
	\Statex{ $\quad\quad$ If a criterion is satisfied (see \ref{notes}) update $\mu$:}
	\Statex{ $\quad\quad\quad\quad$ $\mu^{k+1}=\mu^k/\omega$, for some factor $\omega>1$}
	\Statex{ $\quad\quad\quad\quad$ $\bflambda^{k+1}=\bflambda^k$} 
	\Statex{ $\quad\quad$ Else update $\bflambda$:}
	\Statex{ $\quad\quad\quad\quad$ $\mu^{k+1}=\mu^k$}
	\Statex{ $\quad\quad\quad\quad$ $\bflambda^{k+1}=\mu^k\bfr(\bfx^k)$} 
	\Statex{ End}
    \end{algorithmic}
    \label{alg_penaltyIdiot}
\end{algorithm}

The version of the Idiot crash implemented in \clp has several additional features. At the start of the crash, the approximate component-wise minimization is performed twice for each component of $\bfx$. If a $10\%$ decrease in primal infeasibility is not observed in about $30$ iterations, it is considered that the Idiot crash would not be beneficial and the simplex algorithm is started from the origin. If a $10\%$ reduction of the primal infeasibility is observed, then the mechanism for approximate minimization is adjusted. During each subsequent iteration, the function $h(\bfx)$ is minimized for each component $105$ times. There is no indication why this particular value was chosen. However, one of the features is the possibility to decrease this number, so to minimize for each component fewer than $105$ times. From the $50^{\textnormal{th}}$ minimization onwards a check is performed after the function is minimized $10$ times for each component. Progress is measured with a moving average of expected progress. If it is considered that not enough progress is being made, the function is not minimized any longer for the same values of the parameters. Instead, one of $\mu$ or $\bflambda$ is updated and the next iteration is performed. Thus, in the cases when it is likely that the iteration would not be beneficial, not much unnecessary time is spent during it. Another feature is that in some cases there is a limit on the step size for the update of each $x_j$. Additionally, there is a statistical adjustment of the values of $\bfx$ at the end of each iteration. These features are omitted from this paper since experiments showed that they have little or no effect on performance. Depending on the problem size and structure, the weight parameter ($\mu$) is updated either every $3$ iterations, or every $6$. Again, there is no indication why these values are chosen.  To a large extent it must be assumed that the algorithm has been tuned to achieve a worthwhile outcome when possible, and terminate promptly when it is not. The dominant computational cost corresponds to a matrix-vector product with the matrix $A$ for each set of minimizations over the components of $\bfx$.

\subsubsection*{Relation to augmented Lagrangian and quadratic penalty function methods}

In form, the augmented Lagrangian function~(\ref{eq:AL_for_NLP}) and Idiot function~(\ref{eq:IdiotFunction}) are identical for LP problems and in both methods the penalty parameter $\mu$ is reduced over a sequence of iterations. However, they differ fundamentally in the update of $\bflambda$. For the Idiot crash, new values of $\bflambda$ are given by $\bflambda^{k+1}=\mu^k\bfr(\bfx^k)$. Since $\mu$ is reduced to around machine precision and the aim is to reduce $\bfr(\bfx)$ to zero, the components of $\bflambda$ become small. Contrast this with the values of $\bflambda$ in the augmented Lagrangian method, as set out in Algorithm~\ref{alg_AL_nlp}. These are updated by the value $\mu^k\bfr(\bfx^k)$ and converge to the (generally non-zero) Lagrange multipliers for the equations. 

In the Idiot algorithm, when the values of $\bflambda$ are updated, the linear and quadratic functions of the residual $\bfr(\bfx)$ in the Idiot function~(\ref{eq:AL_for_NLP}) are respectively $\mu^k\bfr(\bfx^k)^T\bfr(\bfx)$ and $(2\mu^k)^{-1}\bfr(\bfx)^T\bfr(\bfx)$. Thus, since the values of $\mu^k$ are soon significantly less than unity, the linear term becomes relatively negligible. In this way the Idiot objective function reduces to the quadratic penalty function~(\ref{eq:QuadraticPenaltyFunction_for_NLP}) and the later behaviour of the Idiot crash is akin to that of a simple quadratic penalty method.

\subsection{Preliminary experiments}
\label{subsect:PreliminaryExperiments}

The effectiveness of the Idiot crash is assessed via experiments with \clp (Version 1.16.10), using a set of 30 representative LP test problems in Table~\ref{tab:Probsspeed-upIdiotPercent}. This is the set used by Huangfu and Hall in~\cite{HuHa13}, with \LP{qap15} replacing \LP{dcp2} due to QAP problems being of particular interest and the latter not being a public test problem, and \LP{nug15} replacing \LP{nug12} for consistency with the choice of QAP problems used by Mittelmann~\cite{MittelmannBenchmarksURL}. The three problems \LP{nug15}, \LP{qap12} and \LP{qap15} are linearizations of quadratic assignment problems, where \LP{nug15} and \LP{qap15} differ only via row and column permutations. The experiments in this paper are carried out on a Intel i7-6700T processor rated at 2.80GHz with 16GB of available memory. In all cases the \clp presolve routine is run first, and is included in the total solution times.

\begin{table}
\centerline{
\begin{tabular}{|l|r|r||l|r|r|}\hline Model& Speed-up&Idiot (\%)&Model& Speed-up&Idiot (\%)\\\hline
\LP{cre-b} & 2.6 & 28.9 & \LP{pds-40} & 1.3 & 5.0\\
\bfLP{dano3mip} & 1.4 & 3.6 & \LP{pds-80} & 1.0 & 0.1\\
\bfLP{dbic1} & 1.5 & 40.6 & \LP{pilot87} & 1.3 & 2.5\\
\LP{dfl001} & 1.0 & 0.1 & \bfLP{qap12} & 2.5 & 0.6\\
\LP{fome12} & 1.1 & 0.1 & \bfLP{qap15} & 4.0 & 0.1\\
\LP{fome13} & 1.9 & 3.3 & \bfLP{self} & 6.1 & 22.7\\
\LP{ken-18} & 1.0 & 0.7 & \LP{sgpf5y6} & 1.4 & 4.8\\
\LP{l30} & 1.9 & 1.4 & \bfLP{stat96v4} & 1.7 & 1.2\\
\bfLP{Linf\_520c} & 9.4 & 8.2 & \LP{storm\_1000} & 4.5 & 0.8\\
\bfLP{lp22} & 1.4 & 1.9 & \LP{storm-125} & 4.1 & 10.1\\
\bfLP{maros-r7} & 0.9 & 7.8 & \LP{stp3d} & 6.5 & 0.9\\
\LP{mod2} & 1.4 & 2.7 & \LP{truss} & 0.8 & 17.1\\
\LP{ns1688926} & 1.4 & 1.0 & \LP{watson\_1} & 1.8 & 8.9\\
\bfLP{nug15} & 4.2 & 0.1 & \LP{watson\_2} & 1.1 & 4.4\\
\LP{pds-100} & 2.5 & 5.4 & \LP{world} & 1.3 & 2.0\\
\hline
\end{tabular}}
\caption{Test problems, percentage of solution time accounted for by the Idiot crash, and the speed-up of the primal simplex solver when the Idiot crash is used}
\label{tab:Probsspeed-upIdiotPercent}
\end{table}

To assess the effectiveness of the Idiot crash in speeding up the \clp primal simplex solver over all the test problems, total solution times were first recorded when running \clp with the \texttt{-primals} option. This forces \clp to use the primal simplex solver but makes no use of the Idiot crash. To compare these with total solution time when \clp uses the primal simplex solver following the Idiot crash, it was necessary to edit the source code so that \clp is forced to use the Idiot crash and primal simplex solver. However, otherwise, it ran as in its default state. The relative total solution times are given in the columns in Table~\ref{tab:Probsspeed-upIdiotPercent} headed "Speed-up". The geometric mean speed-up is 1.9, demonstrating clearly the general value of the Idiot crash for the \clp primal simplex solver. Although the Idiot crash is of little or no value (speed-up below 1.25) for seven of the 30 problems, for only two of these problems does it lead to a small slow-down. However, for ten of the 30 problems the speed-up is at least 2.5, a huge improvement as a consequence of using the Idiot crash. The columns headed "Idiot (\%)" give the percentage of the total solution time accounted for by the Idiot crash, the mean value being 6.2\%. For five problems the percentage is ten or more, and this achieves a handsome speed-up in three cases. However, it does include \LP{truss}, for which the Idiot crash takes up 17\% of an overall solution time which is 20\% more than when using the vanilla primal simplex solver. For only this problem can the Idiot crash be considered a significant and unwise investment. Of the ten problems where the Idiot crash results in a speed-up of at least 2.5, for only three does it account for at least ten percent of the total solution time. Indeed, for five of these problems the Idiot crash is no more than one percent of the total solution time. 

This remarkably cheap way to improve the performance of the primal simplex solver is not always of value to \clp since, when it is run without command line options (other than the model file name), it decides whether to use its primal or dual simplex solver. When the former is used, \clp uses problem characteristics to decide whether to use the Idiot crash and, if used, to set parameter values for the algorithm. Default \clp chooses the primal simplex solver (and always performs the Idiot crash) for just the ten LP problems whose name is given in bold type. For half of these problems there is a speed-up of at least 2.5, so the Idiot crash contributes significantly to the ultimate performance of \clp. 
However, for five problems (\LP{cre-b}, \LP{pds-100}, \LP{storm-125}, \LP{storm\_1000} and \LP{stp3d}), the Idiot crash yields a primal simplex speed-up of at least 2.5 but, when free to choose, \clp uses its dual simplex solver. In each case the dual simplex solver is at least as fast as using the primal simplex solver following the Idiot crash, the geometric mean superiority being a factor of 4.0, so the choice to use the dual simplex solver is justified. 

Further evidence of the importance of the Idiot crash to the performance of \clp is given in Table~\ref{tab:Mittelmann}, which gives the solution times from the Mittelmann benchmarks~\cite{MittelmannBenchmarksURL} for the three major commercial simplex solvers and \clp when applied to five notable problem instances. When solving \LP{Linf\_520c}, \clp is vastly faster than the three commerical solvers. For the three QAP linearisations (\LP{nug15}, \LP{qap12} and \LP{qap15}) problems, \clp is very much faster than \cplex. Finally, for \LP{self}, \clp is significantly faster than the commerical solvers.

\begin{table}
\centerline{
\begin{tabular}{|l|rrr|r|}\hline
Model & \cplex & \gurobi & \xpress & \clp\\\hline
\LP{Linf\_520c} & 495 & 1057 & 255 & 35\\
\LP{nug15} & 338 & 14 & 7 & 14\\
\LP{qap12} & 26 & 1 & 1 & 5\\
\LP{qap15} & 365 & 14 & 6 & 13\\
\LP{self} & 18 & 12 & 15 & 5\\
\hline
\end{tabular}}
\caption{Performance of \cplex-12.8.0, \gurobi-7.5.0, \xpress-8.4.0 and \clp-1.16.10 on five notable problem instances from the Mittelmann benchmarks (29/12/17)~\cite{MittelmannBenchmarksURL}}
\label{tab:Mittelmann}
\end{table}

\begin{table}
\centerline{
\begin{tabular}{|l|ll||l|ll|}\hline
Model&Residual&Objective&Model&Residual&Objective\\\hline
\nullbox{12pt}{0pt}
\LP{cre-b} & $\expten{1.3}{-9}$ & $\expten{6.1}{-2}$ & \LP{pds-40} & $\expten{7.0}{-8}$ & $\expten{3.0}{-2}$\\
\LP{dano3mip} & $\expten{6.1}{-10}$ & $\expten{2.0}{-2}$ & \LP{pds-80} & $\expten{2.2}{-7}$ & $\expten{3.4}{-1}$\\
\LP{dbic1} & $\expten{3.8}{-1}$ & $\expten{8.5}{-2}$ & \LP{pilot87} & $\expten{2.1}{0}$ & $\expten{6.8}{-1}$\\
\LP{dfl001} & $\expten{1.1}{-9}$ & $\expten{3.7}{-3}$ & \LP{qap12} & $\expten{3.6}{-10}$ & $\expten{1.7}{-1}$\\
\LP{fome12} & $\expten{6.4}{-9}$ & $\expten{4.3}{-3}$ & \LP{qap15} & $\expten{2.1}{-10}$ & $\expten{2.8}{-3}$\\
\LP{fome13} & $\expten{1.2}{-8}$ & $\expten{5.2}{-3}$ & \LP{self} & $\expten{5.7}{-5}$ & $\expten{2.4}{-3}$\\
\LP{ken-18} & $\expten{5.4}{-8}$ & $\expten{7.1}{-2}$ & \LP{sgpf5y6} & $\expten{4.0}{-10}$ & $\expten{2.1}{-1}$\\
\LP{l30} & $\expten{1.1}{-9}$ & $\expten{3.9}{0}$ & \LP{stat96v4} & $\expten{3.0}{-3}$ & $\expten{1.0}{0}$\\
\LP{Linf\_520c} & $\expten{1.1}{-1}$ & $\expten{9.1}{-3}$ & \LP{storm\_1000} & $\expten{5.9}{-6}$ & $\expten{5.9}{-2}$\\
\LP{lp22} & $\expten{1.1}{-9}$ & $\expten{1.3}{-3}$ & \LP{storm-125} & $\expten{1.4}{0}$ & $\expten{1.2}{-1}$\\
\LP{maros-r7} & $\expten{4.0}{-9}$ & $\expten{2.3}{-5}$ & \LP{stp3d} & $\expten{7.0}{-5}$ & $\expten{2.7}{-2}$\\
\LP{mod2} & $\expten{3.9}{0}$ & $\expten{2.1}{-1}$ & \LP{truss} & $\expten{7.1}{-1}$ & $\expten{3.2}{-1}$\\
\LP{ns1688926} & $\expten{2.5}{-9}$ & $\expten{4.8}{5}$ & \LP{watson\_1} & $\expten{7.7}{-6}$ & $\expten{8.7}{-1}$\\
\LP{nug15} & $\expten{2.1}{-10}$ & $\expten{3.7}{-4}$ & \LP{watson\_2} & $\expten{1.4}{-10}$ & $\expten{9.7}{-1}$\\
\LP{pds-100} & $\expten{7.6}{-10}$ & $\expten{3.7}{-4}$ & \LP{world} & $\expten{4.3}{0}$ & $\expten{5.5}{-1}$\\
\hline
\end{tabular}}
\caption{Residual and objective error following the Idiot crash in \clp}
\label{tab:ProbsRsduObjEr}
\end{table}

To asses the limiting behaviour of the Idiot crash as a means of finding a point which is both feasible and optimal, \clp was run with the \texttt{-idiot 200} option using the modified code which forces the Idiot crash to be used on all problems. The results are given in Table~\ref{tab:ProbsRsduObjEr}, where the columns headed ``Residual'' contains the final values of $\|A\bfx-\bfb\|_2$. The columns headed ``Objective'' contains values of $(f-f^*)/\max\{1,|f^*|\}$ as a measure of how relatively close the final value of $f$ is to the known optimal value $f^*$, referred to below as the objective error. This measure of optimality is clearly of no practical value, since $f^*$ is not known. However, it is instructive empirically, and motivates later theoretical analysis. The geometric mean of the residuals is $\expten{1.2}{-6}$ and the geometric mean of the objective error measures is $\expten{6.1}{-2}$.

\begin{figure}
\centerline{\includegraphics[height=10cm]{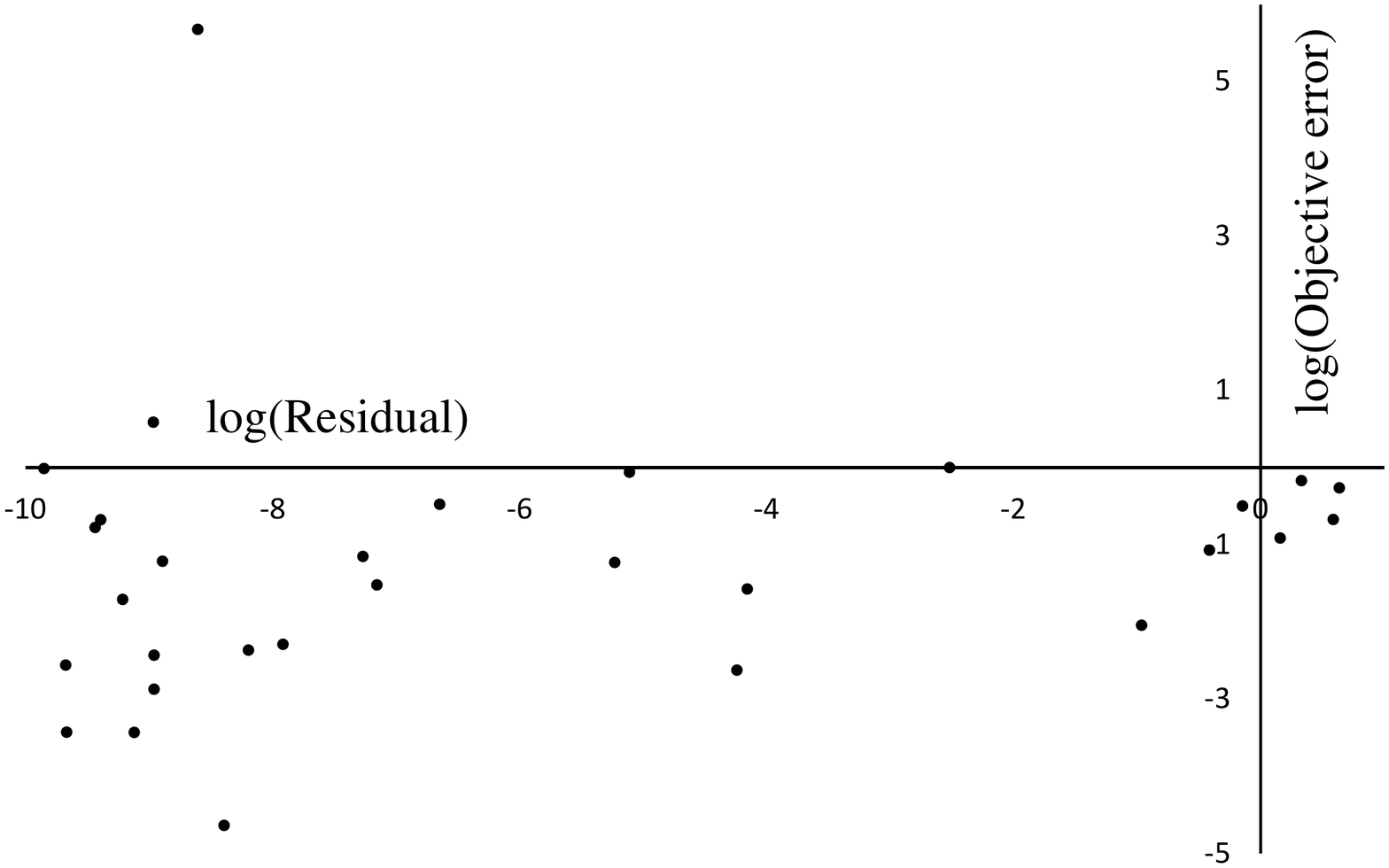}}
\vspace{-1.5cm}
\caption{Distribution of residual and objective errors}
\label{fig:ResObjEr}
\end{figure}

For 17 of the 30 problems in Table~\ref{tab:ProbsRsduObjEr}, the norm of the final residual is less than $10^{-7}$. Since this is the default primal feasibility tolerance for the \clp simplex solver, the Idiot crash can be considered to have obtained an acceptably feasible point. Amongst these problems, the objective error ranges between $\expten{4.8}{5}$ for \LP{ns1688926} and $\expten{2.3}{-5}$ for \LP{maros-r7}, with only eight problems having a value less than $10^{-2}$. Thus, even if the Idiot crash yields a feasible point, it may be far from being optimal. A single quality measure for the point returned by the Idiot crash is convenient, and this is provided by the product of the residual and objective error. As illustrated by the distribution of the objective errors and residual in Figure~\ref{fig:ResObjEr}, it is unsurprising that there are no problems for which a low value of this product corresponds to an accurate optimal objective function value but large residual. 

The observations resulting from the experiments above yield three questions which merit further study. Firstly, since the Idiot crash yields a near-optimal solution for some problems, to what extent does the Idiot crash possess theoretical optimality and convergence properties? Secondly, since the Idiot crash performs particularly well for some problems and badly for others, what problem features might characterise this behaviour? Thirdly, for any problem class where the Idiot crash appears to perform well, might this be valuable? These questions are addressed in the remainder of this paper.

\subsection{Analysis}

In analysing the Idiot algorithm, the initial focus is the Idiot function~(\ref{eq:IdiotFunction}). Fully expanded, this is the quadratic function
$$
h(\bfx) = \dfrac{1}{2\mu} \bfx^TA^TA\bfx + (\bfc^T+\bflambda^TA-\dfrac{1}{\mu} \bfb^TA)\bfx - \bflambda^T\bfb + \dfrac{1}{2\mu}\bfb^T\bfb.
$$
Although convexity of the function follows from the Hessian matrix $A^TA$ being positive semi-definite, it has rank $m<n$. However, the possibility of unboundedness of $h(\bfx)$ on $\bfx\ge\bfzero$ can be discounted as follows. Firstly, observe that unboundedness could only occur in non-negative directions of zero curvature so they must satisfy $A\bfd=\bfzero$. Hence $h(\bfx+\alpha\bfd)=h(\bfx)+\alpha\bfc^T\bfd$ which, if unbounded below for increasing $\alpha$, implies unboundedness of the LP along the ray $\bfx+\alpha\bfd$ from any point $\bfx\ge\bfzero$ satisfying $A\bfx=\bfb$. Thus, so long as the LP is neither infeasible, nor unbounded, $h(\bfx)$ is bounded below on $\bfx\ge\bfzero$. \\

For some problems, the size of the residual and objective measures in Table~\ref{tab:ProbsRsduObjEr} indicate that Idiot has found a point which is close to being optimal. It is, therefore, of interest to know whether the Idiot crash possess theoretical optimality and convergence properties. The Idiot algorithm with approximate minimization of the Idiot function~(\ref{eq:IdiotFunction}) is not conducive to detailed mathematical analysis. However, Theorem~\ref{Theorem_1} shows that if the Idiot function is minimized exactly and an optimal solution to the LP exists, every limit point of the sequence $\left\{\bfx^k\right\}$ is a solution to the problem.

\subsubsection{Notes} 
\label{notes}
\begin{itemize}
\item During each iteration of the loop at most one of the two parameters $\mu$ and $\bflambda$ is updated: in \clp, $\mu$ is updated once every few (e.g. $3$ or $6$) iterations. How often $\mu$ is updated does not affect the validity of the proof as long as $\left\{\mu^k\right\} \to 0$ as $k \to \infty$ and $\bflambda$ is updated at least once every $W$ iterations for some constant $W$.\\

\item In the statement of Algorithm \ref{alg_penaltyIdiot} it is said that the constant $\omega$ is larger than $1$. This is not required for the proof, which would still hold in the case of non-monotonicity of  $\left\{\mu^k\right\}$ as long as $\left\{\mu^k\right\} \to 0$ as $k \to \infty$.\\
\end{itemize}

\begin{theorem}
\label{Theorem_1}
Suppose, that $\bfx^k$ is the exact global minimizer of $h^k(\bfx)$ for each $k=1,2...$ and that $\left\{\mu^k\right\} \to 0$ as $k \to \infty$. Then every limit point of the sequence $\left\{\bfx^k\right\}$ is a solution to the problem \eqref{eq:standardForm}.\\
\end{theorem}

\begin{proof}
\noindent Let $\bar{\bfx}$ be a solution of \eqref{eq:standardForm} so, for all feasible $\bfx$, $ \bfc^T \bar{\bfx} \leq \bfc^T\bfx$. For each $k$, $\bfx^k$ is the exact global minimizer of 

\begin{equation}
\begin{array}{rl}
\label{eqProof}
\displaystyle \min_{\bfx} &h^k(\bfx) = \bfc^T \bfx + {\bflambda^k}^T\bfr(\bfx) + \dfrac{1}{2\mu^k}\bfr(\bfx)^T\bfr(\bfx) \\
\textrm{s.t.} & \bfx \geq \bf0,\\
\end{array}
\end{equation}

\noindent and, since $\bar{\bfx}$ is feasible for \eqref{eq:standardForm}, it is also feasible for \eqref{eqProof}. Thus, since $h^k(\bfx^k) \leq  h^k(\bar{\bfx})$ for each $k$, it follows that 

\begin{align}
	\bfc^T \bfx^k + {\bflambda^k}^T\bfr(\bfx^k) + \dfrac{1}{2\mu^k}\bfr(\bfx^k)^T\bfr(\bfx^k) &\leq  \bfc^T \bar{\bfx} + {\bflambda^k}^T\bfr(\bar{\bfx}) + \dfrac{1}{2\mu^k}		\bfr(\bar{\bfx})^T\bfr(\bar{\bfx}).
	\label{eqStar}
\end{align}

\noindent Since $\bar{\bfx}$ is a solution of~\eqref{eq:standardForm}, $\bfr({\bar{\bfx}})=0$, so
\eqref{eqStar} simplifies to


\begin{align}
	\bfc^T \bfx^k + {\bflambda^k}^T\bfr(\bfx^k) + \dfrac{1}{2\mu^k}\bfr(\bfx^k)^T\bfr(\bfx^k) &\leq  \bfc^T \bar{\bfx} \nonumber \\
    \implies\qquad
	 \dfrac{1}{2\mu^k}\bfr(\bfx^k)^T\bfr(\bfx^k) &\leq   \bfc^T \bar{\bfx} - \bfc^T \bfx^k - {\bflambda^k}^T\bfr(\bfx^k)\nonumber  \\
    \implies\qquad\qquad\!
	\bfr(\bfx^k)^T\bfr(\bfx^k) &\leq  2\mu^k \big ( \bfc^T \bar{\bfx} - \bfc^T \bfx^k - {\bflambda^k}^T\bfr(\bfx^k) \big ).
	\label{eqLimit1}
\end{align}

\noindent At the end of the previous iteration of the loop in Algorithm \ref{alg_penaltyIdiot}, one of the two parameters $\mu$ and $\bflambda$ was updated. If during the previous iteration the update was of $\bflambda$, then $$\bflambda^{k}=\mu^{k-1}\bfr(\bfx^{k-1}).$$

\noindent Alternatively, during the previous iteration $\mu$ was updated and $\bflambda$ remained unchanged, so $$\bflambda^{k}=\bflambda^{k-1}.$$ Consider iterations $k-W\ldots k-1$ of the loop and let $p$ be the index of the latest iteration when the value of $\bflambda$ was changed. Then for some $p$ satisfying $k-W \leq p < k$,

 $$\bflambda^{k}=\mu^{p}\bfr(\bfx^{p}).$$

\noindent Suppose that $\bfx^*$ is a limit point of $ \left\{ \bfx^k   \right\}$, so that there is an infinite subsequence $\mathcal{K}$ such that 
$$\displaystyle \lim_{k \in \mathcal{K}} \bfx^k = \bfx^*.$$

\noindent Taking the limit in inequality \eqref{eqLimit1},

\begin{equation}
\displaystyle \lim_{k \in \mathcal{K}} \bfr(\bfx^k)^T\bfr(\bfx^k) \leq  \displaystyle \lim_{k \in \mathcal{K}} 2\mu^k \big ( \bfc^T \bar{\bfx} - \bfc^T \bfx^k - {\bflambda^k}^T\bfr(\bfx^k) \big).
\label{eqLimit2}
\end{equation}

\noindent For all $k>W$ there is an index $p$ with $k-W \leq p < k$  and $\bflambda^{k}=\mu^{p}\bfr(\bfx^{p})$ so the value of $\bflambda^{k}$ can be substituted in \eqref{eqLimit2} to give

\begin{align*}
\displaystyle \lim_{k \in \mathcal{K}} \bfr(\bfx^k)^T\bfr(\bfx^k)  &\leq  \displaystyle \lim_{k \in \mathcal{K}} 2\mu^k \big ( \bfc^T \bar{\bfx} - \bfc^T \bfx^k - {\bflambda^k}^T\bfr(\bfx^k) \big ) \\
    \implies\qquad
\bfr(\bfx^*)^T\bfr(\bfx^*) &=  \displaystyle \lim_{k \in \mathcal{K}} 2\mu^k \big ( \bfc^T \bar{\bfx} - \bfc^T \bfx^k - {\mu^{p}\bfr(\bfx^{p})}^T\bfr(\bfx^k) \big) \\
    \implies\qquad\qquad\!\!
\|\bfr(\bfx^*)\|^2 &=  \displaystyle \lim_{k \in \mathcal{K}} 2\mu^k( \bfc^T \bar{\bfx} - \bfc^T \bfx^k) - \displaystyle \lim_{k \in \mathcal{K}} 2\mu^k {\mu^{p}\bfr(\bfx^{p})}^T\bfr(\bfx^k)  = 0,
\end{align*}
since $\left\{\mu^k\right\} \to 0$  for $k \in \mathcal{K}$, so $A\bfx^*=\bfb$. For each $\bfx^k$, $\bfx^k \geq \bf0$, so after taking the limit $\bfx^* \geq \bf0$. Thus, $\bfx^*$ is feasible for \eqref{eq:standardForm}. To show optimality of $\bfr(\bfx^*)$, from \eqref{eqStar} 

\begin{align}
	\bfc^T \bfx^k + {\bflambda^k}^T\bfr(\bfx^k) + \dfrac{1}{2\mu^k}\bfr(\bfx^k)^T\bfr(\bfx^k) &\leq  \bfc^T \bar{\bfx} \nonumber \\
    \implies\qquad
\displaystyle \lim_{k \in \mathcal{K}} \big ( \bfc^T \bfx^k + {\bflambda^k}^T\bfr(\bfx^k) + \dfrac{1}{2\mu^k}\bfr(\bfx^k)^T\bfr(\bfx^k) \big ) &\leq \displaystyle \lim_{k \in \mathcal{K}} \bfc^T \bar{\bfx} \nonumber \\
    \implies\qquad\!\!\!\!
\bfc^T \bfx^* +\displaystyle \lim_{k \in \mathcal{K}} {\bflambda^k}^T\bfr(\bfx^k) +\displaystyle \lim_{k \in \mathcal{K}} \dfrac{1}{2\mu^k}\bfr(\bfx^k)^T\bfr(\bfx^k) &\leq  \bfc^T \bar{\bfx} 
\label{eqLimit3}
\end{align}

\noindent For all $k>W$ there is an index $p$ with $k-W \leq p < k$  and $\bflambda^{k}=\mu^{p}\bfr(\bfx^{p})$ so

\begin{equation*}
\displaystyle \lim_{k \in \mathcal{K}} {\bflambda^k}^T\bfr(\bfx^k)   =  \displaystyle \lim_{k \in \mathcal{K}} {\mu^{p}\bfr(\bfx^{p})}^T\bfr(\bfx^k) = 0, 
\end{equation*}

\noindent since $ \left\{\mu^k\right\} \to 0$ for $k = 1,2\ldots$ and $p \to \infty$ as $k \to \infty$. This value can be substituted in \eqref{eqLimit3} giving  
\begin{equation*}
\bfc^T \bfx^*  +\displaystyle \lim_{k \in \mathcal{K}} \dfrac{1}{2\mu^k}\bfr(\bfx^k)^T\bfr(\bfx^k) \leq  \bfc^T \bar{\bfx}.
\end{equation*}

\noindent For each $k$, $\mu^k > 0$ and $\bfr(\bfx)^T\bfr(\bfx) \geq 0$ for each $\bfx$, so

$$\dfrac{1}{2\mu^k}\bfr(\bfx^k)^T\bfr(\bfx^k) \geq 0 \quad  \forall k \implies 
\bfc^T \bfx^* \leq \bfc^T \bfx^*  +\displaystyle \lim_{k \in \mathcal{K}} \dfrac{1}{2\mu^k}\bfr(\bfx^k)^T\bfr(\bfx^k) \leq  \bfc^T \bar{\bfx}. $$

\noindent Consequently, $\bfx^*$ is feasible for \eqref{eq:standardForm} and has an objective value less than or equal to the optimal value $\bfc^T\bar{\bfx}$ so $\bfx^*$ is a solution of \eqref{eq:standardForm}.

\end{proof}

\section{Fast approximate solution of LP problems}\label{sect:FASLP}

Although Theorem~\ref{Theorem_1} establishes an important ``best case'' result for the behaviour of the Idiot crash, the results in Table~\ref{tab:ProbsRsduObjEr} show that this is far from being representative of its practical performance. For some problems it yields a near-optimal point; for others it terminates at a point which is far from being feasible. What problem characteristics might explain this behaviour and, if it is seen to perform well for a whole class of problems, to what extent is this of further value?

\subsection{Problem characteristics affecting the performance of the Idiot crash}

There is a clear relation between the condition number of the matrix $A$ and the solution error of the point returned by the Idiot crash. Of the problems in Table~\ref{tab:ProbsRsduObjEr}, all but \LP{storm\_1000} are sufficiently small for the condition of $A$ (after the \clp presolve) to be computed with the resources available to the authors. These values are plotted against the solution error in Figure~\ref{fig:SolErVsLPCond}, which clearly shows that the problems solved accurately have low condition number. Notable amongst these are the QAPs which, with the exception of \LP{maros-r7}, have very much the smallest condition numbers of the 29 problems in Table~\ref{tab:ProbsRsduObjEr} for which condition numbers could be computed.

Nocedal and Wright~\cite[p.512]{NoWr06} observe that ``there has been a resurgence of interest in penalty methods, in part because of their ability to handle degenerate problems''. However, analysis of optimal basic solutions of the problems in Table~\ref{tab:ProbsRsduObjEr} showed no meaningful correlation between their primal or dual degeneracy and accuracy of the point returned by the Idiot crash.

\begin{figure}
\centerline{\includegraphics[height=10cm]{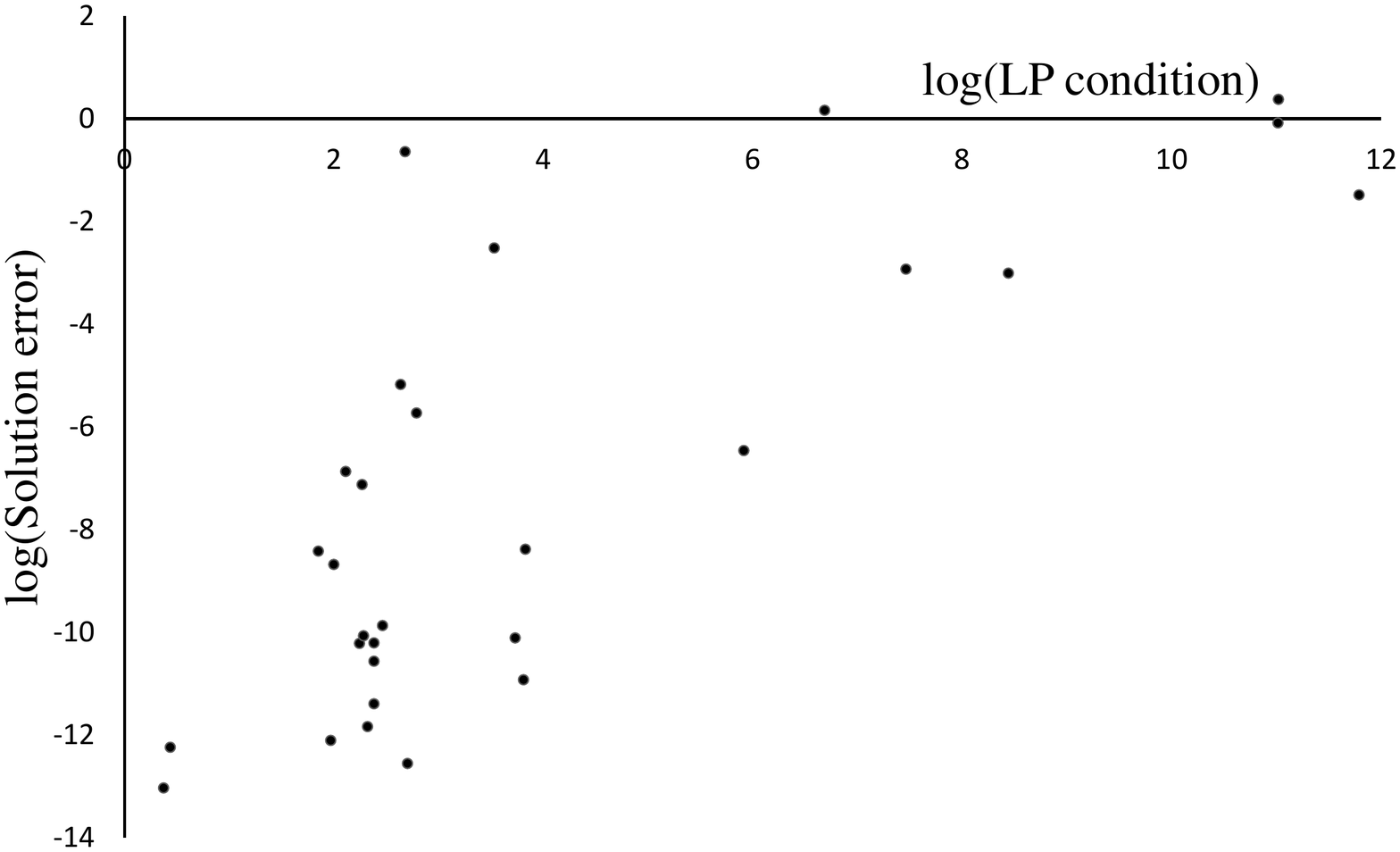}}
\vspace{-1.5cm}
\caption{Solution error and LP condition}
\label{fig:SolErVsLPCond}
\end{figure}

\subsection{The Idiot crash on QAPs}

Since the Idiot crash yields a near-optimal point for the three QAPs in Table~\ref{tab:ProbsRsduObjEr}, it is of interest to know the extent to which this behaviour is typical of the whole class of such problems, and its practical value. Both of these issues are explored in this section.

\subsubsection*{Quadratic assignment problems}
The quadratic assignment problem (QAP) is a combinatorial optimization problem, being a special case of the facility location problem. It concerns a set of facilities and a set of locations. For each pair of locations there is a distance, and for each pair of facilities there is a weight or flow specified, for instance the number of items transported between the two facilities. The problem is to assign all facilities to different locations so that the sum of the distances multiplied by the corresponding flows is minimized. QAPs are well known for being very difficult to solve, even for small instances. They are NP-hard and the travelling salesman problem can be seen as a special case. Often, rather than solve the quadratic problem an equivalent linearization is solved. A comprehensive survey of QAP problems and their solution is given by Loiola~\etal~\cite{LoAbEtAl07}. 

The test problems \LP{nug15}, \LP{qap12} and \LP{qap15} referred to above are examples of the Adams and Johnson linearization~\cite{AdJo94}. Although there are many specialised techniques for solving QAP problems, and alternative linearizations, the popular Adams and Johnson linearization is known to be hard to solve using the simplex method or interior point methods~\cite{ReRaDr95}. Table~\ref{tab:IdiotQAPs} gives various performance measures for the Idiot crash when applied to the Nugent~\cite{NuVoRu68} problems, using the default iteration limit of \clp. The first of these is the value of the residual $\|A\bfx-\bfb\|_2$ at the point obtained by the Idiot crash. For the smaller problems this is significantly larger due to the Idiot crash choosing to perform relatively few iterations for these problems. However, for the larger problems the point is clearly feasible to within the \clp simplex tolerance. The objective function value and relative error are also given, and in most cases the latter is within 1\%. It is not known why the errors for \LP{nug15} and \LP{nug20} are significantly larger than for the other large instances. Finally, the time for Idiot is given. Whilst this is growing, Idiot clearly obtains a near-optimal solution for QAP instances \LP{nug20} and \LP{nug30} which cannot be solved with commercial simplex or interior point implementations on the machine used for the Idiot experiments due to excessive time or memory requirements. 

\begin{table}
\centerline{
\begin{tabular}{|l|rrr|r|rr|r|}\hline
Model&Rows&Columns&Optimum&Residual&Objective&Error&Time\\\hline
\nullbox{12pt}{0pt}
\LP{nug05} & 210 & 225 & 50.00 & $\expten{1.5}{-3}$ & 54.10 & $\expten{-7.6}{-2}$ & 0.01\\
\LP{nug06} & 372 & 486 & 86.00 & $\expten{2.7}{-3}$ & 87.17 & $\expten{-1.3}{-2}$ & 0.02\\
\LP{nug07} & 602 & 931 & 148.00 & $\expten{5.3}{-3}$ & 150.37 & $\expten{-1.6}{-2}$ & 0.03\\
\LP{nug08} & 912 & 1613 & 203.50 & $\expten{8.2}{-3}$ & 207.51 & $\expten{-1.9}{-2}$ & 0.06\\
\LP{nug12} & 3192 & 8856 & 522.89 & $\expten{3.6}{-10}$ & 523.87 & $\expten{-1.9}{-3}$ & 7.19\\
\LP{nug15} & 6330 & 22275 & 1041.00 & $\expten{4.4}{-9}$ & 1265.14 & $\expten{-1.8}{-1}$ & 12.88\\
\LP{nug20} & 15240 & 72600 & 2182.00 & $\expten{2.8}{-9}$ & 2852.50 & $\expten{-2.4}{-1}$ & 41.32\\
\LP{nug30} & 52260 & 379350 & 4805.00 & $\expten{1.1}{-10}$ & 4811.38 & $\expten{-1.3}{-3}$ & 212.62\\
\hline
\end{tabular}}
\caption{Performance of the Idiot crash on QAP linearizations}
\label{tab:IdiotQAPs}
\end{table}

There is currently no practical measure of the point obtained by the Idiot crash which gives any guarantee that it can be taken as a near-optimal solution of the problem. The result of Theorem~\ref{Theorem_1}
cannot be used since the major iteration minimization is approximate, and the major iterations are terminated rather than being performed to the limit. Clearly the measure of objective error in Table~\ref{tab:IdiotQAPs} requires knowledge of the optimal objective function value. What can be guaranteed, however, is that since the point returned is feasible, the corresponding objective value is an upper bound on the optimal objective function value. With the aim of identifying an interval containing the optimal objective function value, the Idiot crash was applied to the dual of the linearization. Although it obtained points which were feasible for the dual problems to within the \clp simplex tolerance, the objective values were far from being optimal so the lower bounds thus obtained were too weak to be of value. 

\section{Conclusions}\label{sect:Conclusions}

Forrest's aim in developing the Idiot crash for LP problems was to determine a point which, when used to obtain a starting basis for the primal revised simplex method, results in a significant reduction in the time required to solve the problem. This paper has distilled the essence of the Idiot crash and presented it in algorithmic form for the first time. Practical experiments have demonstrated that, for some large scale LP test problems, Forrest's aim is achieved. For LP problems when the Idiot crash is not advantageous, this is identified without meaningful detriment to the performance of \clp. For the best case in which the Idiot sub-problems are solved exactly, Theorem~\ref{Theorem_1} shows that every limit point of the sequence of Idiot iterations is a solution of the corresponding LP problem. It is observed empirically that, typically, the lower the condition of the constraint matrix $A$, the closer the point obtained by the Idiot crash is to being an optimal solution of the LP problem. For linearizations of quadratic assignment problems, it has been demonstrated that the Idiot crash consistently yields near-optimal solutions, achieving this in minutes for instances which are intractable on the same machine using commercial LP solvers. Thus, in addition to achieving Forrest's initial aim, the Idiot crash is seen as being useful in its own right as a fast solver for amenable LP problems.

\bibliographystyle{plain}
\bibliography{AL}


\end{document}

%% file: alias.tex
\newcommand{\clp}{\texttt{Clp}\xspace}
\newcommand{\cbc}{\texttt{Cbc}\xspace}
\newcommand{\xpress}{\texttt{Xpress}\xspace}

\newcommand{\cplex}{\texttt{Cplex}\xspace}
\newcommand{\gurobi}{\texttt{Gurobi}\xspace}

\def\bfb{\mbox{\boldmath$b$}}
\def\bfc{\mbox{\boldmath$c$}}
\def\bfd{\mbox{\boldmath$d$}}
\def\bfx{\mbox{\boldmath$x$}}

\def\bfr{\mbox{\boldmath$r$}}

\def\bfzero{\mbox{\boldmath$0$}}

\def\bflambda{\mbox{\boldmath$\lambda$}}

\def\R{\mathbb{R}}

\newcommand{\LP}[1]{\textsc{#1}\xspace}
\newcommand{\bfLP}[1]{\bf\textsc{#1}\xspace}
\newcommand{\expten}[2]{#1\!\times\!10^{\textrm{#2}}}
\def\nullbox#1#2{\setbox0=\null\ht0=#1 \dp0=#2\box0}

\def\etal{{\em et al.\ }}

\newcommand{\ignore}[1]{}